\documentclass[11pt,letterpaper,reqno]{amsart}
\usepackage{amsmath}
\usepackage{amsfonts}
\usepackage{amssymb}
\usepackage{graphicx}

\usepackage[dvipsnames]{xcolor}
\usepackage[english]{babel}
\usepackage{amsthm}
\usepackage{enumitem}
	

\allowdisplaybreaks
\newtheorem{theorem}{Theorem}
\newtheorem{corollary}{Corollary}
\newtheorem{proposition}{Proposition}
\newtheorem{lemma}{Lemma}

\newtheorem{remark}{Remark}

\usepackage[colorlinks=true, allcolors=blue]{hyperref}
\AtBeginDocument{%
  \setlength{\abovedisplayskip}{8pt plus 2pt minus 2pt}%
  \setlength{\belowdisplayskip}{8pt plus 2pt minus 2pt}%
  \setlength{\abovedisplayshortskip}{4pt plus 2pt minus 1pt}%
  \setlength{\belowdisplayshortskip}{4pt plus 2pt minus 1pt}%
}
\begin{document}

\author{S. Dann}
\thanks{The first author is supported by the FAPA funds from Vicerrectoría de Investigaciones de la Universidad de los Andes (INV-2019-63-1699).}

\author{O. Herscovici}

\author{S. Myroshnychenko}
\thanks{The third author is supported in part by NSERC RGPIN-2024-05044.}

\address{S. Dann, Departamento de Matemáticas, Universidad de los Andes, Carrera 1 \#18A -- 12
Bogot\'a, 111711, Colombia}
\email{s.dann@uniandes.edu.co}

\address{O. Herscovici, Department of Mathematics and Computer Science, St. John's University, 8000
Utopia Parkway, Queens, NY 11439, USA}
\email{herscovo@stjohns.edu}

\address{S.~Myroshnychenko, Department of Mathematics and Statistics,  University of the Fraser Valley, 33844 King Rd, Abbotsford, BC V2S 7M7,  Canada}
\email{serhii.myroshnychenko@ufv.ca}

\title[Flotation Rigidity for Convex Polytopes]{Flotation Rigidity for Convex Polytopes}

\begin{abstract}
We prove that the oriented liquid-level function determines
a full-dimensional convex polytope $P\subset\mathbb R^d$,
$d\ge2$, and its relative density
$\delta\in(0,1)\setminus\{\tfrac12\}$.
Uniqueness also follows from buoyancy surfaces at distinct
densities accumulating at a non-half interior value.
For polygons, fixed-density buoyancy data determine
the normalized difference body, yielding uniqueness among centrally
symmetric polygons at every density.

\end{abstract}

\maketitle

\tableofcontents

\section{Introduction}

Questions about the unique determination of convex bodies
from their sections or projections lie at the intersection
of classical mechanics, harmonic analysis, differential
geometry, and convex geometry
\cite{ARSY1,ARSY2,BMO,HSW,MR2018,O,R23}.
In this work, we study the determination of convex polytopes
from their flotation and buoyancy surfaces.
Interest in these problems increased further after the construction in
\cite{R22} of a nonspherical convex body in \(\mathbb R^3\), of relative
density \(\tfrac12\), that floats in equilibrium in every orientation,
providing a negative answer to Ulam's question. Our results establish
rigidity for oriented flotation data at one non-half density and for
buoyancy surfaces over an interval of densities.

To describe our results, we briefly review the setting. Let
\(K\subset\mathbb R^d\), \(d\ge2\), be a convex body of uniform density
\(\rho_K\) floating in a liquid of uniform density \(\rho>0\), where
\(0<\rho_K<\rho\). Archimedes' principle \cite{Ar} states that the submerged
volume \(V_{\mathrm{sub}}\) and the $d$-dimensional volume of the body $|K|_d$ 

satisfy
\[
\rho V_{\mathrm{sub}}=\rho_K|K|_d.
\]
After normalizing the liquid density to \(1\), we write
\(\delta=\rho_K/\rho\in(0,1)\), so
\(V_{\mathrm{sub}}=\delta|K|_d\).

For each unit vector \(\theta \in \mathbb{S}^{d-1}\), there exists a unique
real number \(t_K(\theta,\delta)\) such that the liquid surface
$
H(\theta,t_K(\theta,\delta))
=
\{x \in \mathbb{R}^d : x \cdot \theta = t_K(\theta,\delta)\}
$
determines the submerged half-space
\[
H^+(\theta,t_K(\theta,\delta))
=
\{x \in \mathbb{R}^d : x \cdot \theta \ge t_K(\theta,\delta)\},
\]
for which
\[
|K \cap H^+(\theta,t_K(\theta,\delta))|_d
=
\delta |K|_d.
\]
The map $\theta\mapsto t_K(\theta,\delta)$ is the
\emph{oriented liquid-level function}.

\begin{figure}[h!]
    \centering
    \includegraphics[width=0.75\linewidth]{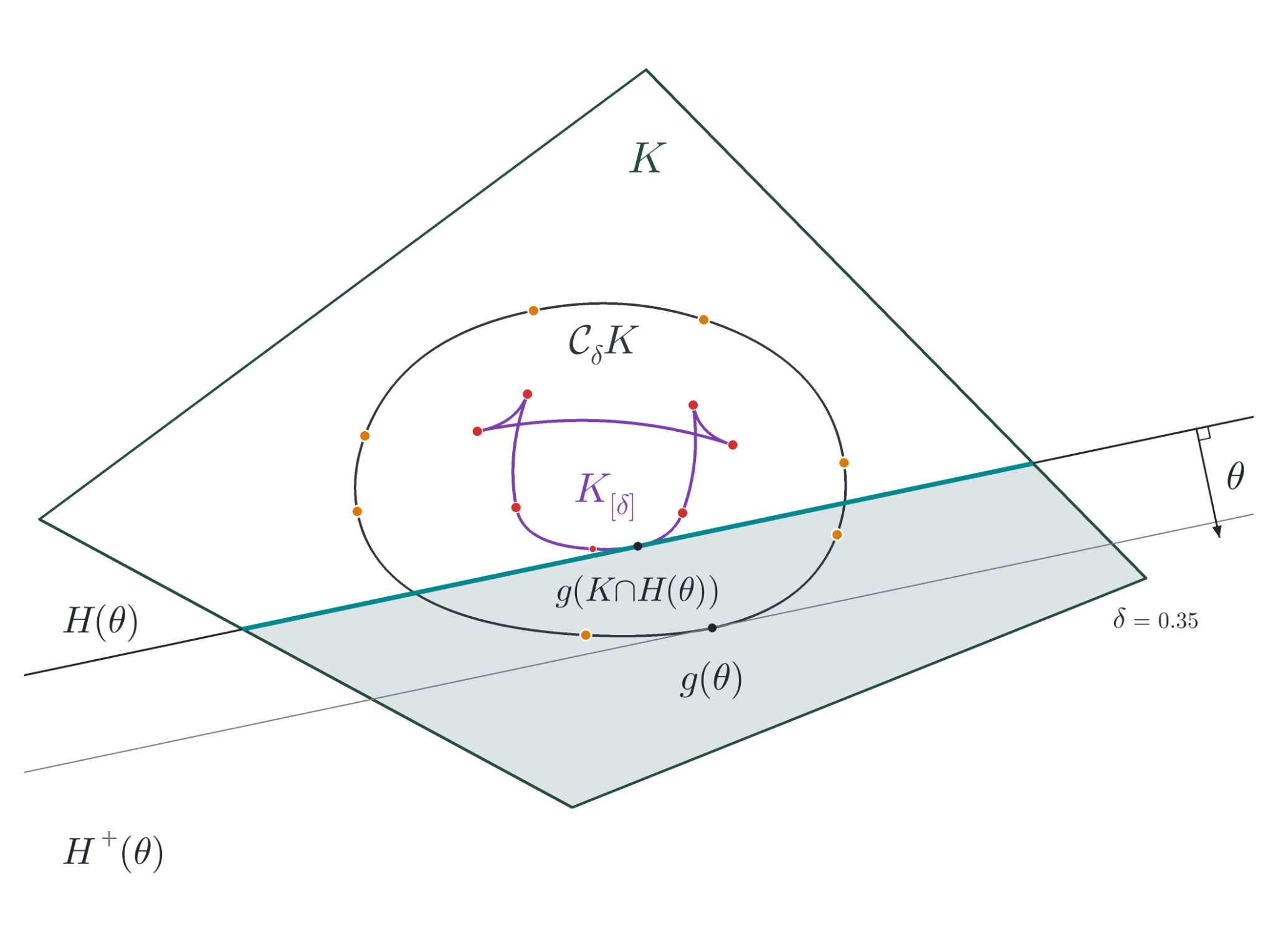}
    \caption{Surface of flotation $K_{[\delta]}$ and buoyancy surface $\mathcal{C}_\delta K$ for a convex body~$K$.}
    \label{fig:FS}
\end{figure}

For the submerged cap
$
K \cap H^+(\theta,t_K(\theta,\delta)),
$
let \(g(\theta)\) denote its centroid. As \(\theta\) varies over the
unit sphere, the image of the map \(g\) traces a hypersurface
\(\mathcal{C}_\delta K\) in \(\mathbb{R}^d\), called the
\emph{surface of centres} (or \emph{buoyancy surface}) of \(K\),
see Figure~\ref{fig:FS}.

The envelope \(K_{[\delta]}\) of the truncating hyperplanes determined
by \(t_K(\theta,\delta)\) is the classical \emph{surface of flotation}
of \(K\). This surface has been extensively studied in affine convex
geometry and is closely related to affine surface area, affine
isoperimetric inequalities, and convex floating bodies
\cite{Du,SW}. By the Second Theorem of Dupin \cite{R23,Z}, the
contact point between \(K_{[\delta]}\) and the hyperplane
\(H(\theta)\) is the centroid \(g(K \cap H(\theta))\) of the section
\(K \cap H(\theta)\). For centrally symmetric \(K\) and \(0<\delta<\tfrac12\), the
flotation surface \(K_{[\delta]}\) is the boundary of the convex floating
body \(K_\delta\) \cite{MR,W}. Densities above \(\tfrac12\) are covered
by the complementary-density identity stated below.

Notice 
that if \(K\) is centrally symmetric, then \(K_{[\frac12]}\) reduces to a single point, that is, the centre of symmetry of \(K\). Consequently, one cannot expect uniqueness from the flotation surface at density \(\delta=\frac12\). 
We prove that the oriented liquid-level function determines both a
convex polytope and its density whenever at least one of the compared
densities differs from $\tfrac12$.
\begin{theorem}\label{floatingTHM}
    Let $P$ and $Q$ be full-dimensional convex polytopes in $\mathbb{R}^{d}$, $d \geq 2$, with uniform densities $\delta_1, \delta_2 \in (0,1)$, respectively. If at least one of $\delta_1, \delta_2$ differs from $\frac{1}{2}$ and
    $
    t_P(\theta,\delta_1) = t_Q(\theta, \delta_2)
    $ for all $\theta \in \mathbb S^{d-1}$,
    then $P = Q$ and $\delta_1 = \delta_2$.
\end{theorem}

We note that by the construction, $K_{[\delta]} = K_{[1-\delta]}$.

In dimension $d=2$, the proof of Theorem \ref{floatingTHM} is based on considerations of the piecewise-hyperbolic structure of the flotation curve \cite{VP,Z}. In higher dimensions, the analytic structure of the oriented liquid-level function allows us to recover the vertices of the polytope.
Theorem~\ref{floatingTHM} also yields uniqueness from
buoyancy surfaces at distinct densities accumulating at
an interior value different from $\tfrac12$.

\begin{corollary}
\label{thm:buoyancy-density-interval}
Let $P,Q\subset\mathbb R^d$, $d\ge2$, be full-dimensional convex
polytopes. Let $(\delta_j)_{j\ge1}\subset(0,1)$ be a sequence of
distinct densities converging to
\(
\delta_0\in(0,1)\setminus\{\tfrac12\}.
\)
If
\[
\mathcal C_{\delta_j}P=\mathcal C_{\delta_j}Q
\qquad\text{for every }j\ge1,
\]
then $P=Q$.
\end{corollary}
In particular, equality of the buoyancy surfaces on any nonempty
open interval of densities determines the polytope.

For the uniform probability measure on a convex body, the
convex body bounded by its buoyancy surface is a zonoid
trimmed region (see \cite[Proposition~5.1 and equation~(6.4)]{MT2021}).
These regions determine the underlying distribution when
known for all parameters in either $(0,\tfrac12]$ or
$[\tfrac12,1)$ \cite[Theorem~5.6]{KM1997}.
In the polyhedral class, Theorem~\ref{floatingTHM} yields
determination from any nonempty open interval of densities.
The idea, discussed in Section~\ref{sec:buoyancy-density-interval},
is to recover the oriented liquid-level function by differentiating support values with respect to the density.

At a single prescribed density, we obtain a planar reconstruction
result. 

\begin{corollary}
\label{diff-body}
Let $P,Q\subset\mathbb R^2$ be full-dimensional convex
polygons. If $\mathcal C_\delta P=\mathcal C_\delta Q$ for a fixed
$\delta\in(0,1)$, then
\[
\frac{P-P}{|P|_2^{1/3}}
=
\frac{Q-Q}{|Q|_2^{1/3}},
\]
where $P-P=\{x-y:x,y\in P\}$ is the difference body.
    
\end{corollary}
Consequently, the buoyancy transform $\mathcal C_\delta$ is
injective on the class of centrally symmetric convex polygons
at every prescribed density, including $\delta=\tfrac12$,
without any assumption of equal areas.
The proof combines the section-moment identity with the
straightness of polygonal sides.
The remaining fixed-density reconstruction question is also discussed
in Section~\ref{sec:buoyancy-uniqueness}.

This work contributes to the broader program of identifying nonlinear geometric transforms that exhibit rigidity phenomena. 
It raises further questions concerning stability, extensions beyond the polyhedral class,  connections with the affine surface area, valuations, and intrinsic volumes~\cite{S}.

\section{Preliminaries}
We work in the ambient Euclidean space $\mathbb{R}^d$, $d \geq 2$, with the usual orthonormal basis $\{e_1, \ldots, e_d\}$. Given $x=(x_1, \ldots,x_d)$ and $y=(y_1, \ldots, y_d)$ in $\mathbb{R}^d$, their scalar product is $x\cdot y = \sum_{j=1}^d x_jy_j$, and we write $\|x\|=(x\cdot x)^{1/2}$ for the Euclidean norm of $x$. The unit sphere in $\mathbb{R}^d$ is $\mathbb{S}^{d-1} = \{x \in \mathbb{R}^d: \, \|x\|=1\}.$ 
For a differentiable function $h$ on $\mathbb{S}^{d-1}$, $\nabla_S h$ stands for its spherical gradient, that is, the projection of the Euclidean gradient of any differentiable local extension of $h$ onto the tangent space of $\mathbb S^{d-1}$.
For a subset $S \subset \mathbb{R}^d$ and $\lambda \in \mathbb{R}$, its homothetic copy is
$
\lambda S = \{\lambda x: \, x \in S\},
$
where $\lambda x = (\lambda x_1,\lambda x_2, \ldots, \lambda x_d)$. Let $\partial S$ denote the boundary of $S$, and let $\operatorname{int}S$ denote its interior. We also use $\operatorname{conv}S$ to denote the convex hull of $S$. For a subspace $H \subset \mathbb{R}^d$, $H^\perp$ stands for its orthogonal complement.
For a differentiable map
$F:\mathbb S^{d-1}\to\mathbb R^k$, $u\perp\theta$, 
we denote
\[
DF(\theta)[u]
=\left.\frac{d}{ds}
F\left(\frac{\theta+su}{\|\theta+su\|}\right)\right|_{s=0}.
\]

A body $K \subset \mathbb{R}^d$ is a compact subset of $\mathbb{R}^d$ with non-empty interior. A body $K$ is convex if for every pair of points $x,y \in K$, the closed segment 
\[
[x,y] = \{\lambda x + (1-\lambda)y: \, 0 \leq \lambda \leq 1 \}
\]
is contained in $K$. 
We are interested in a special family of convex bodies -- namely, convex polytopes, which can be realized as the convex hull of a finite set of points in $\mathbb{R}^d$.
We denote by \(\operatorname{vert}P\) its set of vertices (extreme points), namely the points of \(P\) that cannot be expressed as a nontrivial convex combination of two distinct points of \(P\).

The centroid of a body $L \subset \mathbb{R}^k$ is 
\begin{equation*} \label{centre}
    g(L) = \frac{1}{|L|_k} \int_L x \, dx,
\end{equation*}
where $| \cdot |_k$ stands for the $k$-dimensional Lebesgue measure (volume). 
Let $\mathcal H^{d-1}$ denote the $(d-1)$-dimensional Hausdorff measure normalized
to agree with Euclidean volume on hyperplanes.

Instead of rotating the body within the liquid, we fix $K$ and rotate the liquid surface $H(\theta)$. These liquid surfaces determine the liquid-level function
$t_K(\theta,\delta)$ introduced above. 

\section{Uniqueness of flotation surfaces \texorpdfstring{$P_{[\delta]}$}{P[delta]}}
To prove Theorem \ref{floatingTHM}, we first show that $P_{[\delta]}$ is a hedgehog generated by a \(C^1\), piecewise analytic liquid-level function, \cite{MM}.

\begin{lemma}\label{analyticity-lemma}
Let $P\subset\mathbb{R}^{d}$, $d\geq2$, be a full-dimensional convex polytope with
uniform density $\delta\in(0,1)$, and put
\(
h(\theta)=t_P(\theta,\delta), \,\,\theta\in\mathbb{S}^{d-1}.
\)
For every vertex $v\in\operatorname{vert}P$, define
$$
Z_{P,v}
=
\{\theta\in\mathbb{S}^{d-1}:
v\cdot\theta=h(\theta)\},
$$
and set
\(
\Sigma_P
=
\bigcup_{v\in\operatorname{vert}P}Z_{P,v}.
\)
Then:

\begin{enumerate}[label=(\roman*)]
\item
$h\in C^1(\mathbb{S}^{d-1})$.

\item
$h$ is real analytic on
$\mathbb{S}^{d-1}\setminus\Sigma_P$.  Consequently,
\[
\operatorname{Sing}(h)\subseteq\Sigma_P,
\]
where $\operatorname{Sing}(h)$ denotes the set of points
at which $h$ is not real analytic.

\item
For every $v\in\operatorname{vert}P$, the set $Z_{P,v}$ is a
nonempty compact $C^1$ submanifold of $\mathbb{S}^{d-1}$ of
dimension $d-2$.

\item
If
\[
\theta_0\in Z_{P,v}
\setminus
\bigcup_{\substack{u\in\operatorname{vert}P\\u\neq v}}
Z_{P,u},
\]
that is, if the liquid hyperplane corresponding to $\theta_0$
contains exactly one vertex $v$ of $P$, then $h$ is not real
analytic at $\theta_0$.
\end{enumerate}
\end{lemma}

\begin{proof}
For $\theta\in\mathbb{S}^{d-1}$, write
\(
H_P(\theta)
=
\{x\in\mathbb{R}^{d}:x\cdot\theta=h(\theta)\},
\)
and
\(
S_P(\theta)
=
P\cap H_P(\theta).
\)
Let $c_P(\theta)=g(S_P(\theta))$ denote the centroid of the
$(d-1)$-dimensional section $S_P(\theta)$.
Since $\delta\in(0,1)$, every hyperplane (the liquid's surface) intersects
$\operatorname{int}P$, and therefore
\(
|S_P(\theta)|_{d-1}>0.
\)
Introduce the cap-volume function
\[
V_P(\theta,t)
=\left|P\cap\{x\in\mathbb R^d:x\cdot\theta\ge t\}\right|_d.
\]
At interior cuts, section volumes and moments depend
continuously on the cutting hyperplane. 
Fubini's theorem and dominated convergence  give continuously
differentiable cap volumes and first moments; see
\cite[proof of Theorem~1.2, equations~(2.5)--(2.6)]{HSW}.
This argument also applies when the cutting hyperplane
contains vertices. Thus $V_P\in C^1$ at interior cuts and,
for $u\perp\theta$,
\[
\begin{aligned}
\frac{\partial V_P}{\partial t}(\theta,t)
&=-|P\cap H(\theta,t)|_{d-1}<0, \qquad
D_\theta V_P(\theta,t)[u]
=\int_{P\cap H(\theta,t)}x\cdot u\,
  d\mathcal H^{d-1}(x).
\end{aligned}
\]
By definition,
\begin{equation}\label{eq:cap-volume}
V_P(\theta,h(\theta))=\delta|P|_d.
\end{equation}
The \(C^1\) Implicit Function Theorem now gives \(h\in C^1\) and
\[
Dh(\theta)[u]
=\frac{1}{|S_P(\theta)|_{d-1}}
 \int_{S_P(\theta)}x\cdot u\,d\mathcal H^{d-1}(x)
=c_P(\theta)\cdot u.
\]
Since \(c_P(\theta)\cdot\theta=h(\theta)\), this is equivalent to
Dupin's identity \cite{R23,Z},
\begin{equation} \label{DI}
    c_P(\theta)=h(\theta)\theta+\nabla_S h(\theta).
\end{equation}
This proves~(\textit{i}); the same argument applies to any convex body \(P\).

We next prove~(\textit{ii}).
Suppose $\theta_0\notin\Sigma_P$.  Then
$H_P(\theta_0)$ contains no vertex of $P$.  Hence, after restricting
to a sufficiently small neighborhood of
$(\theta_0,h(\theta_0))$, every vertex of $P$ remains strictly in the same half-space of the moving hyperplane
\(
\{x:x\cdot\theta=t\}.
\)
Consequently, the combinatorial type of the truncated polytope is
constant in this neighborhood.

If an edge $[a,b]$ intersects the moving hyperplane, its point of
intersection is
\[
a+
\frac{t-a\cdot\theta}
     {(b-a)\cdot\theta}(b-a).
\]
Since the endpoints $a$ and $b$ remain on opposite sides of the
hyperplane, the denominator does not vanish.  Thus every vertex of
the truncated polytope is a real analytic function of $(\theta,t)$.
Using a fixed triangulation of the truncated polytope, it follows
that $V_P(\theta,t)$ is real analytic near
$(\theta_0,h(\theta_0))$.

Moreover, recall
\(
\frac{\partial V_P}{\partial t}
(\theta,h(\theta))
=
-|S_P(\theta)|_{d-1}<0.
\)
Therefore the analytic Implicit Function Theorem \cite{KP} applied to
\eqref{eq:cap-volume} shows that $h$ is real analytic in a
neighborhood of $\theta_0$.  This proves~(2).

We now prove~(\textit{iii}).  Fix a vertex
$v\in\operatorname{vert}P$ and define
\(
F_v(\theta)
=
h(\theta)-v\cdot\theta.
\)
Then
\(
Z_{P,v}=F_v^{-1}(0).
\)
At $\theta\in Z_{P,v}$, both $v$ and $c_P(\theta)$ belong to
$H_P(\theta)$, and hence
\(
(c_P(\theta)-v)\cdot\theta=0.
\)
Using Dupin's identity \eqref{DI}, we obtain
\(
\nabla_S F_v(\theta)
=
c_P(\theta)-v.
\)
Now $v$ is an extreme point of the section $S_P(\theta)$, whereas
the centroid $c_P(\theta)$ belongs to
$\operatorname{relint}S_P(\theta)$.  Hence
\(
c_P(\theta)\neq v,
\)
and therefore
\(
\nabla_SF_v(\theta)\neq0.
\)
Thus $0$ is a regular value of $F_v$, and $Z_{P,v}$ is a closed
$C^1$ submanifold of $\mathbb{S}^{d-1}$ of codimension one \cite{Hi}, i.e., $(d-2)$-dimensional.  Since $\mathbb{S}^{d-1}$ is compact,
$Z_{P,v}$ is compact.

It remains to prove that $Z_{P,v}$ is nonempty.  Choose
\(
\eta\in\operatorname{int}N_P(v) \cap \mathbb{S}^{d-1},
\)
where $N_P(v)$ is the normal cone of $P$ at $v$, with the convention
that
\[
(v-x)\cdot\eta>0,
\qquad
\text{for all }x\in P\setminus\{v\}.
\]
The hyperplane through $v$ perpendicular to $\eta$ supports $P$ at
$v$, and therefore
\[
\left|
P\cap\{x:x\cdot\eta\geq v\cdot\eta\}
\right|_d=0.
\]
For the direction $-\eta$, the corresponding half-space contains
all of $P$, so
\[
\left|
P\cap\{x:x\cdot(-\eta)\geq v\cdot(-\eta)\}
\right|_d
=
|P|_d.
\]
Since $d\geq2$, there exists a continuous path on
$\mathbb{S}^{d-1}$ joining $\eta$ to $-\eta$.  By continuity of the
cut-off volume, along this path there exists
$\theta\in\mathbb{S}^{d-1}$ such that
\(
\left|
P\cap\{x:x\cdot\theta\geq v\cdot\theta\}
\right|_d
=
\delta|P|_d.
\)
Hence
\(
h(\theta)=v\cdot\theta,
\)
and therefore $\theta\in Z_{P,v}$.  This proves~(\textit{iii}).

We finally prove~(\textit{iv}), using a cap-volume version of the
vertex-crossing argument in
\cite[proof of Theorem~2.1, higher-dimensional case]{Yaskin2011}.
Let
\[
\theta_0\in Z_{P,v}\setminus
\bigcup_{\substack{u\in\operatorname{vert}P\\u\neq v}}Z_{P,u},
\qquad t_0=h(\theta_0).
\]
Thus $H_P(\theta_0)$ contains $v$ and no other vertex, so
\begin{equation}\label{eq:vertices-separated}
u\cdot\theta_0\neq t_0,
\qquad u\in\operatorname{vert}P\setminus\{v\}.
\end{equation}
Choose a star triangulation of $P$ at $v$,
\[
P=\bigcup_{k=1}^N T_k,\qquad
T_k=\operatorname{conv}(v,v_{k,1},\ldots,v_{k,d}), \qquad v_{k,j} \in \operatorname{vert}P \setminus \{v\}.
\]
By \eqref{eq:vertices-separated}, in a sufficiently small
neighborhood $\mathcal W$ of $(\theta_0,t_0)$, only $v$ can
cross the moving hyperplane $x\cdot\theta=t$.
For a simplex $T=\operatorname{conv}(v,v_1,\ldots,v_d)$,
the standard truncation formula
\cite[Theorem~2.2]{Lasserre2015} gives, for generic $\theta$,
\[
\left|T\cap\{x:x\cdot\theta\geq t\}\right|_d
=
|T|_d\sum_{i=0}^d
\frac{(\alpha_i(\theta)-t)_+^d}
{\prod_{j\neq i}(\alpha_i(\theta)-\alpha_j(\theta))}, \qquad \text{where} \quad y_+ = \max\{y,0\},
\]
where $\alpha_0(\theta)=v\cdot\theta$ and
$\alpha_j(\theta)=v_j\cdot\theta$.
For generic directions, all terms except the one corresponding
to $v$ retain the same branch as $t$ crosses $v\cdot\theta$. The changing term is
\[
\frac{|T|_d\,(v\cdot\theta-t)_+^d}
{\prod_{j=1}^d((v-v_j)\cdot\theta)}.
\]
Thus, summing over the simplices, the two analytic
cap-volume branches $V_+$ and $V_-$ satisfy
\begin{equation}\label{eq:branch-difference-final}
V_+(\theta,t)-V_-(\theta,t)
=A_v(\theta)(v\cdot\theta-t)^d,
\end{equation}
up to interchanging the branches, where
\begin{equation}\label{eq:Av-final}
A_v(\theta)=
\sum_{k=1}^N
\frac{|T_k|_d}
{\prod_{j=1}^d((v-v_{k,j})\cdot\theta)}.
\end{equation}
Here $A_v$ is the combined coefficient of the local
$d$-th order volume change at $v$.

To justify the identity at nongeneric directions, choose a fixed
triangulation of each truncated simplex separately on the two
sides of $t=v\cdot\theta$. All vertices in these triangulations
extend analytically to $\mathcal W$. Indeed, the denominators
in the edge--hyperplane intersection formulas are nonzero:
for edges incident to $v$, this follows from
\eqref{eq:vertices-separated}; for other cut edges, their
endpoints remain strictly on opposite sides of the hyperplane.
Each simplex volume on the chosen side has the form
\[
\frac{\varepsilon}{d!}
\det(p_1-p_0,\ldots,p_d-p_0),
\]
where its vertices $p_i$ depend analytically on $(\theta,t)$
and $\varepsilon\in\{-1,1\}$ is fixed. These expressions remain
analytic even when some simplices collapse at
$t=v\cdot\theta$. Summing them therefore gives analytic
extensions of both cap-volume branches to $\mathcal W$.
Since $A_v$ is also analytic there,
\eqref{eq:branch-difference-final} extends from generic
directions to all of $\mathcal W$ by continuity.

We need only show $A_v\not\equiv0$.
Choose a generic 
$\eta\in\operatorname{int}N_P(v)$. Then
\(
(v-u)\cdot\eta>0, u\in\operatorname{vert}P\setminus\{v\},
\)
so every summand in \eqref{eq:Av-final} is positive.
Consequently,
\begin{equation}\label{eq:Av-positive}
A_v(\eta)>0.
\end{equation}

Suppose, toward a contradiction, that $h$ is analytic
near $\theta_0$. Put
\(
\lambda(\theta)=v\cdot\theta-h(\theta).
\)
By \eqref{DI},
\[
\lambda(\theta_0)=0,\qquad
\nabla_S\lambda(\theta_0)
=v-c_P(\theta_0)\neq0,
\]
since $v$ lies on the boundary of the water section
while its centroid lies in its relative interior.
Hence, in a sufficiently small connected neighborhood $U$,
the sets
\[
U_+=\{\theta\in U:\lambda(\theta)>0\},
\qquad
U_-=\{\theta\in U:\lambda(\theta)<0\}
\]
are nonempty and open. Label the branches so that $V_+$ is the actual cap
volume on $U_+$. Then
\(
V_+(\theta,h(\theta))=\delta|P|_d, \theta\in U_+.
\)
By the assumed analyticity of \(h\) and the identity theorem for
real-analytic functions \cite{KP}, this equality holds throughout
the connected neighborhood \(U\).
On $U_-$, however, the actual branch is $V_-$, so
\(
V_-(\theta,h(\theta))=\delta|P|_d.
\)
Subtracting and using
\eqref{eq:branch-difference-final} yields
$
A_v(\theta)\lambda(\theta)^d=0$ for $\theta\in U_-$.
Since $\lambda\neq0$ on $U_-$, we obtain
$A_v=0$ there.

Finally, clearing denominators in \eqref{eq:Av-final}
gives $A_v=N_v/D_v$, with $N_v,D_v$ homogeneous
polynomials. Thus $N_v$ vanishes on $U_-$ and,
by homogeneity, on the open cone
\(
\{r\theta:r>0,\ \theta\in U_-\}\subset\mathbb R^d.
\)
Hence $N_v\equiv0$, so $A_v\equiv0$, contradicting
\eqref{eq:Av-positive}. Therefore $h$ is not real
analytic at $\theta_0$, proving~(4).

\end{proof}

\subsection{\texorpdfstring{Proof of Theorem~\ref{floatingTHM} for $d=2$}{Planar proof of the flotation theorem}}
In dimension two, the flotation surface admits a piecewise-hyperbolic
description, allowing a direct reconstruction argument. In higher
dimensions, we instead exploit the analytic structure of the oriented
liquid-level function and its singular locus.

The following description of the planar flotation curve is related
to Davidov's classical analysis of buoyancy curves of polygonal
prisms \cite{Da}; see also \cite[Section~9.4]{ARSY1}
for an account of this approach.
For a recent differential-geometric treatment of planar flotation
and buoyancy, see \cite{Za}.

\begin{lemma}
\label{P_Dupin}
Let $P\subset\mathbb R^2$ be a convex polygon with nonempty
interior and $\delta\in(0,1)$. Let $H(\varphi)$ be its liquid
line with normal $\theta(\varphi)=(\cos\varphi,\sin\varphi)$,
and let $m(\varphi)$ be the midpoint of its liquid chord.

The normal-parameter circle admits a finite subdivision such
that, on each open interval, the chord joins two fixed side
interiors. If these sides are nonparallel, $m$ traces a
nondegenerate hyperbolic arc with their supporting lines as
asymptotes; if they are parallel, $m$ is constant.

Unless $m$ is constant everywhere, consecutive hyperbolic
pieces in normal-parameter order meet in one of two ways:
\begin{description}
\item[Case 1]
They meet directly, with common limiting tangent
$H(\varphi_0)$.
\item[Case 2]
They are separated by a maximal constant interval
$[\varphi_-,\varphi_+]$, with
$0<\varphi_+-\varphi_-<\pi$.
Its liquid lines pass through their common endpoint, and
the limiting tangents $H(\varphi_-)$ and $H(\varphi_+)$
are distinct.
\end{description}
The whole curve can collapse to a point only when
$\delta=\tfrac12$, and does so for every centrally symmetric
polygon at that density.
\end{lemma}

\begin{proof}
Write $h(\varphi)=t_P(\theta(\varphi),\delta)$.
By Lemma~\ref{analyticity-lemma}, $h\in C^1$, and for each vertex
$v$ the set
\(
Z_{P,v}=\{\theta\in\mathbb S^1:v\cdot\theta=t_P(\theta,\delta)\}
\)
is a compact zero-dimensional submanifold of $\mathbb S^1$,
hence finite. Thus the directions for which a liquid line
contains a vertex form a finite set. 
Between consecutive vertex directions, the chord endpoints
remain in the relative interiors of the same two sides.
By the Second Theorem of Dupin
\cite{R23,Z}, their midpoint is $m(\varphi)$, so $m$ is continuous.

First suppose the two supporting lines are nonparallel.
Choose affine coordinates in which they are the coordinate
axes, and write the chord endpoints as $(a,0)$ and $(0,b)$.
The polygonal area formula gives the prescribed cap area as
\(
C+\lambda ab,
\)
where $C$ and $\lambda\ne0$ are constant on the interval.
Hence $ab\ne0$: neither endpoint can be the intersection
of the supporting lines while it lies in a side interior.
The midpoint therefore traces the nondegenerate hyperbola
\(
XY=\frac{ab}{4}.
\)
At its point $(\frac a 2, \frac b2)$, the tangent line is
\(
bX+aY=ab,
\)
which is precisely the liquid chord line. The asymptotes are
the two supporting lines.

If the sides are parallel, choose orthonormal coordinates in
which they lie on $y=0$ and $y=q$, with $q>0$, and write the
endpoints as $(s,0)$ and $(t,q)$. The polygonal area formula
expresses the cap area as
\[
C+\varepsilon\frac q2(s+t),
\qquad \varepsilon\in\{-1,1\},
\]
with $C$ and $\varepsilon$ fixed on the interval. Therefore
$s+t$ is constant, and so is the midpoint
$(\frac{s+t}{2},\frac q 2)$. This proves the asserted finite decomposition.

The limiting tangent line at an endpoint of a hyperbolic piece
is the limiting liquid line. Continuity of $m$ therefore gives
Case~1 whenever two such pieces meet without an intervening
constant interval. On a constant interval with value $p$,
the identity $m(\varphi)\cdot\theta(\varphi)=h(\varphi)$ shows
that every $H(\varphi)$ passes through $p$; the preceding
decomposition gives the asserted subdivision into parallel-side
families.

It remains to show that a constant interval separating
hyperbolic pieces has angular length less than $\pi$.
Otherwise it contains two parameters $\varphi$ and $\varphi+\pi$,
and
\[
h(\varphi)=p\cdot\theta(\varphi),\qquad
h(\varphi+\pi)=-h(\varphi).
\]
The corresponding opposite liquid half-planes are complementary
and both cut off area $\delta|P|_2$, so $\delta=\tfrac12$.
At this density, uniqueness of the halving level gives, for
every $\psi$,
\[
h(\psi+\pi)=-h(\psi),\qquad m(\psi+\pi)=m(\psi).
\]
Constancy on an interval of length at least $\pi$ would therefore
make $m$ constant everywhere, contrary to the hypothesis.
Thus $0<\varphi_+-\varphi_-<\pi$, and the two endpoint liquid
lines have distinct unoriented normals. This proves Case~2.

The same complementary-cap argument shows that an everywhere
constant curve requires $\delta=\tfrac12$. Finally, if $P$ is
centrally symmetric about $p$, every halving line passes through
$p$ and cuts a centrally symmetric chord with midpoint $p$.
\end{proof}

\begin{remark}\label{rem:antipodal-levels}
For a convex body $K\subset\mathbb R^d$ and $\delta\in(0,1)$,
\[
t_K(-\theta,\delta) = -t_K(\theta,\delta)
\]
holds for some $\theta\in\mathbb S^{d-1}$ if and only if
$\delta=\tfrac12$, in which case it holds for every $\theta$.

Indeed, the displayed equality means that the opposite liquid
hyperplanes coincide. Their complementary caps both have volume
$\delta|K|_d$, so $2\delta=1$. Conversely, at density $\tfrac12$,
uniqueness of the halving level gives the equality in every direction.
\end{remark}

The next lemma gives a billiard-like reconstruction procedure
from the oriented liquid-level function and completes the
planar proof of Theorem~\ref{floatingTHM}.

\begin{lemma}
\label{lem:planar-flotation-reconstruction}
Let $P,Q\subset\mathbb R^2$ be convex polygons with nonempty
interiors and densities $\alpha,\beta\in(0,1)$, respectively.
Suppose that at least one of $\alpha,\beta$ differs from $\tfrac12$ and
\[
t_P(\theta, \alpha) = t_Q(\theta, \beta)
\qquad\text{for every }\theta\in\mathbb S^1.
\]
Then $P=Q$ and $\alpha=\beta$.
\end{lemma}

\begin{proof}
Interchanging $P$ and $Q$ if necessary, assume $\alpha\ne\tfrac12$.
We regard the sides as maximal boundary segments, merging consecutive
collinear segments. Put
\[
\theta(\varphi)=(\cos\varphi,\sin\varphi),\qquad
u(\varphi)=(-\sin\varphi,\cos\varphi),
\]
and write $h(\varphi)$ for the common liquid-level function. Thus
\[
H(\varphi)=\{x:x\cdot\theta(\varphi)=h(\varphi)\}
\]
is the liquid line in both polygons.

By the Second Theorem of Dupin, as used in
Lemma~\ref{P_Dupin}, the point of the normal-parametrized curve
associated with a liquid line is the midpoint of its chord.
Since $h\in C^1$ by Lemma~\ref{analyticity-lemma}, the two polygons
have the common chord midpoint
\begin{equation*}
\label{eq:flotation-common-midpoint}
m(\varphi)
=h(\varphi)\theta(\varphi)+h'(\varphi)u(\varphi).
\end{equation*}
In particular, the midpoints of $P\cap H(\varphi)$ and
$Q\cap H(\varphi)$ coincide. This identity also holds at vertex
directions by continuity.

Consequently, if $x$ is one endpoint of a liquid chord corresponding
to $\varphi$, its other endpoint is
\(
y=2m(\varphi)-x.
\)
For later use, fix a convex polygon $K$ and $\eta\in(0,1)$.
Define $T_{K,\eta}(x)$ as the unique point $y\in\partial K$
such that the counterclockwise boundary arc from $x$ to $y$,
closed by the chord $[y,x]$, encloses area $\eta|K|_2$.
For a fixed $x\in\partial K$, as $y$ traverses the boundary once
counterclockwise from $x$, this area increases continuously from
$0$ to $|K|_2$, strictly whenever its value lies between these two
extremes. Hence the endpoint exists uniquely
and depends continuously on $x$. The complementary arc gives
\(
T_{K,\eta}^{-1}=T_{K,1-\eta},
\)
so $T_{K,\eta}$ is a boundary homeomorphism. Moreover, the two oriented
liquid chords through a boundary point $x$ have other endpoints
$T_{K,\eta}(x)$ and $T_{K,\eta}^{-1}(x)$. The directions of liquid
lines passing through vertices form a finite set by
Lemma~\ref{P_Dupin}.

We first recover chords joining nonparallel sides. Suppose that,
on an open interval of directions, the chord in $P$ joins the
relative interiors of two nonparallel sides with supporting lines
$\ell_1,\ell_2$. By Lemma~\ref{P_Dupin}, the midpoint traces
a nondegenerate hyperbolic arc whose underlying hyperbola has
asymptotes precisely $\ell_1,\ell_2$.

Subdivide this direction interval at the finitely many vertex
directions of $Q$. On each resulting interval the endpoints in $Q$
lie on fixed sides. By Lemma~\ref{P_Dupin}, these sides cannot
be parallel: their midpoint would then be constant, whereas the
common midpoint traces a nonconstant hyperbolic arc.
The same lemma gives a nondegenerate hyperbola for $Q$. Since
the midpoints coincide, this hyperbola shares an open arc with
the one arising from $P$. Two nondegenerate hyperbolas sharing
an open arc coincide, so their asymptotes, and hence the two
supporting lines, coincide. Intersecting them with $H(\varphi)$
shows that the chord endpoints in $P$ and $Q$ are identical.
By continuity this also holds at the subdivision points.

Put $T=T_{P,\alpha}$, and let $\mathcal S$ be the set of points
$x\in\partial P$ such that $x$ and $T(x)$ lie in the relative
interiors of nonparallel sides. The preceding argument gives
\begin{equation}
\label{eq:flotation-seed-points}
\mathcal S\subseteq\partial Q.
\end{equation}

We next show that a chain joining parallel sides has bounded length.
Set $c=\alpha |P|_2$. If consecutive chords in a $T$-orbit
join side interiors and all these side pairs are parallel, the orbit
must alternate between the same two sides: a convex polygon has at
most two maximal sides in any unoriented direction. Denote them by
$E,F$. Choose positively oriented orthonormal coordinates with
$E\subset\{y=0\}$, $F\subset\{y=q\}$, and $q>0$
(Figure~\ref{fig:flotation-parallel-cases}). The separation is
positive because $E$ and $F$ are distinct maximal parallel sides.

\begin{figure}[htbp]
    \centering
    \includegraphics[width=0.65\linewidth]{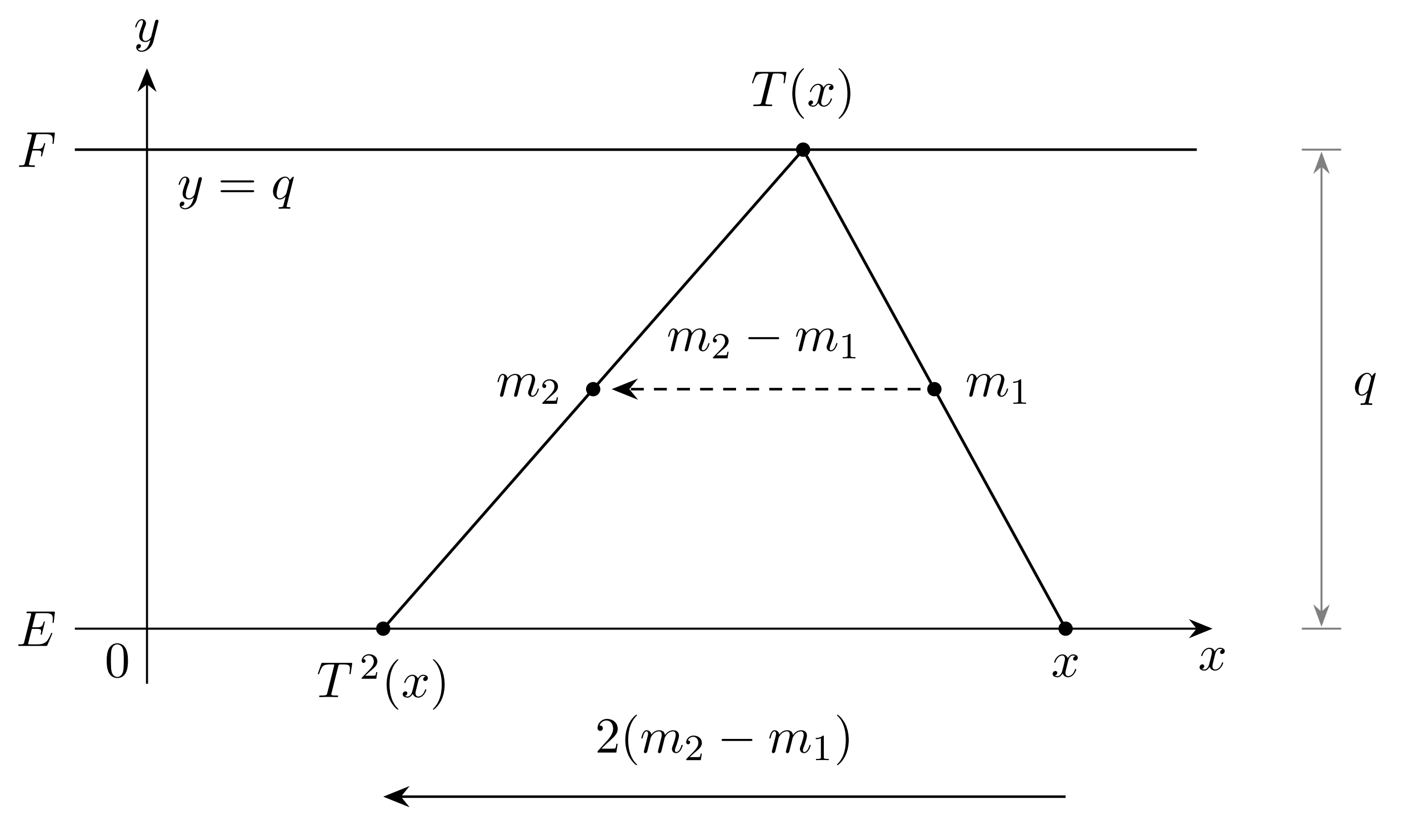}
    \caption{The midpoints $m_1, m_2$ give
$T^2(x)-x=2(m_2-m_1)$.}
    \label{fig:flotation-parallel-cases}
\end{figure}

Write $x=(s,0)$, $T(x)=(t,q)$ and, provided the next chord also joins
these sides, $T^2(x)=(s',0)$. The counterclockwise boundary arc from
$(s,0)$ to $(t,q)$ contains a fixed chain of vertices. By the
determinant formula for polygonal area \cite{Nu}, the area the chain
encloses with its closing chord is
\(
B(s,t)=C-\frac q2(s+t)
\)
for a constant $C$. The two consecutive cap conditions give
\[
B(s,t)=c,\qquad B(s',t)=|P|_2-c.
\]
Subtracting yields
\begin{equation*}
\label{eq:flotation-parallel-drift}
s'-s=\frac{2(2c-|P|_2)}q
=\frac{2(2\alpha-1)|P|_2}q.
\end{equation*}
Thus every two steps translate the endpoint by the same nonzero
vector, of magnitude $\frac{2|1-2\alpha||P|_2}q$, as long as the
chain remains on this parallel pair.

Here is an explicit uniform bound. Let $D=\operatorname{diam}P$, the diameter of $P$, 
and set
\[
J=\left\lfloor
\frac{D^2}{2|1-2\alpha||P|_2}
\right\rfloor+1,
\qquad N=2J.
\]
An orbit $x,T(x),\ldots,T^N(x)$ avoiding vertices cannot consist
entirely of parallel-side chords. Otherwise it would alternate
between a fixed pair $E,F$ and, since $q\le D$, satisfy
\[
|T^N(x)-x|
=J\frac{2|1-2\alpha||P|_2}{q}
\ge J\frac{2|1-2\alpha||P|_2}{D}>D,
\]
contradicting the definition of $D$. Hence such an orbit reaches
$\mathcal S$ at an index less than $N$.

Let $V$ be the finite vertex set of $P$ and define
\[
\mathcal E=\bigcup_{j=0}^{N}T^{-j}(V).
\]
The set $\mathcal E$ is finite because $T$ is a homeomorphism.
For $x\in\partial P\setminus\mathcal E$, choose $k<N$ such that
$x_k:=T^k(x)\in\mathcal S$, and write $x_i=T^i(x)$ for $0\le i\le k$.
By~\eqref{eq:flotation-seed-points}, $x_k\in\partial Q$.
Suppose, working backwards, that $x_{i+1}\in\partial Q$, and let
$H(\varphi_i)$ be the liquid line containing $[x_i,x_{i+1}]$ in $P$.
It is also a liquid line in $Q$. Since $0<\beta<1$, this line meets
$\operatorname{int}Q$, so $x_{i+1}$ is an endpoint of its chord
in $Q$ and the other endpoint is
\(
2m(\varphi_i)-x_{i+1}=x_i.
\)
Induction gives $x\in\partial Q$. We have proved
\(
\partial P\setminus\mathcal E\subseteq\partial Q.
\)
Taking closures yields $\partial P\subseteq\partial Q$.
Hence $P=Q$, and consequently $\alpha=\beta$.
\end{proof}

\subsection{\texorpdfstring{Proof of Theorem~\ref{floatingTHM} for $d\ge3$}{Higher-dimensional proof of the flotation theorem}}

\begin{proof}
Write
\(
h(\theta) = t_P(\theta, \delta_1) = t_Q(\theta, \delta_2).
\)
By Remark~\ref{rem:antipodal-levels}, equality of the oriented
liquid-level functions implies that $\delta_1=\tfrac12$ if and
only if $\delta_2=\tfrac12$. Since at least one density differs
from $\tfrac12$, both do.

We shall use the following elementary fact.  Let
$M\subset\mathbb{S}^{d-1}$ be a connected $C^{1}$ submanifold of
dimension $d-2$.  If
\[
M\subset G_1\cup\cdots\cup G_N,
\]
where every $G_j$ is a great $(d-2)$-sphere, then
\(
M\subset G_j
\)
for at least one $j$.

Indeed, put $S_j=M\cap G_j$.  The sets $S_j$ are closed in $M$ and
cover $M$.  The union of their relative interiors is dense in $M$.
If two different great spheres $G_i,G_j$ occurred on two open pieces
whose closures met at a point $p\in M$, the continuity of the tangent
space of $M$ would imply
$
T_pG_i=T_pM=T_pG_j.
$
But two distinct great $(d-2)$-spheres passing through $p$ have
distinct tangent spaces at $p$.  Hence the closures of the relative
interiors corresponding to distinct great spheres are disjoint.
Since $M$ is connected, only one of them can be nonempty, proving the
claim.

We now show that every vertex of $P$ is also a vertex of $Q$.
Fix
$
v\in\operatorname{vert}P
$
and as before consider
\begin{equation*} \label{Z}
Z_{P,v}
=
\{\theta\in\mathbb{S}^{d-1}:v\cdot\theta=h(\theta)\}.
\end{equation*}
By Lemma~\ref{analyticity-lemma}, $Z_{P,v}$ is a nonempty compact
$C^{1}$-smooth $(d-2)$-dimensional submanifold.  Let $C$ be one of its
connected components.
Suppose, towards a contradiction, that $
v\notin\operatorname{vert}Q.
$
We shall show that $C$ is contained in a finite union of proper great
$(d-2)$-spheres.

For $\theta\in C$, there are two possibilities. First, suppose that $H_P(\theta)$ contains another vertex
$u\in\operatorname{vert}P$, $u\neq v$.  Then
\(
v\cdot\theta
=
h(\theta)
=
u\cdot\theta,
\)
so
\(
(v-u)\cdot\theta=0.
\)
Thus
\(
\theta\in
(v-u)^\perp\cap\mathbb{S}^{d-1}.
\)

Second, suppose that $H_P(\theta)$ contains no vertex of $P$ other
than $v$.  By part~(4) of Lemma~\ref{analyticity-lemma}, the common
function $h$ is not real analytic at $\theta$.  Applying part~(2) of
that lemma to $Q$, the liquid hyperplane $H_Q(\theta)$ must therefore
contain at least one vertex
$
w\in\operatorname{vert}Q.
$
Since
\(
v\cdot\theta
=
h(\theta)
=
w\cdot\theta,
\)
we obtain
$
(v-w)\cdot\theta=0,
$
so
\(
\theta\in
(v-w)^\perp\cap\mathbb{S}^{d-1}.
\)
It follows that
\[
C
\subset
\bigcup_{\substack{u\in\operatorname{vert}P\\u\neq v}}
\left((v-u)^\perp\cap\mathbb{S}^{d-1}\right)
\;\cup\;
\bigcup_{w\in\operatorname{vert}Q}
\left((v-w)^\perp\cap\mathbb{S}^{d-1}\right).
\]
Under the assumption $v\notin\operatorname{vert}Q$, every vector
$v-u$ and $v-w$ occurring here is nonzero.  Hence every member of
this finite union is a genuine great $(d-2)$-sphere.

By the auxiliary fact above, $C$ is contained in one such great
sphere, say
\(
G=a^\perp\cap\mathbb{S}^{d-1}.
\)
But $C$ and $G$ have the same dimension, and $C$ is a submanifold
without boundary.  Therefore $C$ is relatively open in $G$.
It is also compact and hence relatively closed in $G$.
Since $G$ is congruent to $\mathbb S^{d-2}$ and $d\ge3$,
it is connected. Therefore $C=G$.
In particular, if $\theta\in C$, then $-\theta\in C$ as well.
Since $C\subset Z_{P,v}$,
\(
h(\theta)=v\cdot\theta
\)
and
\[
h(-\theta)
=
v\cdot(-\theta)
=
-v\cdot\theta
=
-h(\theta).
\]
As observed at the beginning of the proof, this implies
$\delta_1=\frac12$,
contrary to our assumption. Therefore
$
v\in\operatorname{vert}Q.
$
Since $v$ was arbitrary,
$
\operatorname{vert}P
\subseteq
\operatorname{vert}Q.
$
Interchanging $P$ and $Q$, and using
$\delta_2\neq\frac12$, gives
$
\operatorname{vert}Q
\subseteq
\operatorname{vert}P.
$
Hence
\(
\operatorname{vert}P
=
\operatorname{vert}Q,
\)
and therefore
$
P
=
\operatorname{conv}(\operatorname{vert}P)
=
\operatorname{conv}(\operatorname{vert}Q)
=
Q,
$
so $\delta_1=\delta_2$.
\end{proof}

\section{Uniqueness from accumulating density data}
\label{sec:buoyancy-density-interval}

We now consider buoyancy surfaces at varying densities. Differentiation with respect to the density recovers the
oriented liquid-level function, allowing us to apply
Theorem~\ref{floatingTHM}.
For a convex body $K\subset\mathbb R^d$, consider the support value of the buoyancy surface $\mathcal C_\delta K$ in direction~$\theta$, 
\[
s_K(\theta,\delta)
=\max_{z\in\mathcal C_\delta K}z\cdot\theta,
\qquad
(\theta,\delta)\in\mathbb S^{d-1}\times(0,1).
\]
The following identity is the differential form of the standard
averaged-quantile representation
\cite[equation~(2.4)]{MT2021}.
We include a direct geometric proof.

\begin{lemma}
\label{lem:buoyancy-density-differentiation}
Let $K\subset\mathbb R^d$, $d\ge2$, be a convex body. For every
$\theta\in\mathbb S^{d-1}$, the function
$\delta\mapsto s_K(\theta,\delta)\, \delta$ is continuously differentiable
on $(0,1)$, and
\[
\frac{\partial}{\partial\delta}
\bigl(s_K(\theta,\delta)\, \delta\bigr)
=t_K(\theta,\delta).
\]
\end{lemma}

\begin{proof}
Fix $\theta\in\mathbb S^{d-1}$. The upper cap
$K\cap H^+(\theta,t_K(\theta,\delta))$ maximizes
$\int_E x\cdot\theta\,dx$ among measurable subsets $E\subset K$
of volume $\delta|K|_d$. Since its centroid belongs to
$\mathcal C_\delta K$, it follows that
\[
s_K(\theta,\delta)\, \delta
=\frac{1}{|K|_d}
\int_{K\cap H^+(\theta,t_K(\theta,\delta))}
x\cdot\theta\,dx.
\]

For an interior cutting level $t$, put
\[
\begin{aligned}
A_K(\theta,t)&=|K\cap H(\theta,t)|_{d-1}, \quad
V_K(\theta,t)=|K\cap H^+(\theta,t)|_d,\quad
B_K(\theta,t)=\int_{K\cap H^+(\theta,t)}x\cdot\theta\,dx.
\end{aligned}
\]
The section volume $A_K(\theta,t)$ is continuous and positive
between the two supporting levels. Slicing therefore gives
\[
\frac{\partial V_K}{\partial t}(\theta,t)=-A_K(\theta,t),
\qquad
\frac{\partial B_K}{\partial t}(\theta,t)=-tA_K(\theta,t).
\]
Differentiating
$V_K(\theta,t_K(\theta,\delta))=\delta|K|_d$ yields
\(
\frac{\partial t_K}{\partial\delta}(\theta,\delta)
=-\frac{|K|_d}{A_K(\theta,t_K(\theta,\delta))}.
\)
Consequently,
\[
\begin{aligned}
\frac{\partial}{\partial\delta}
\bigl(s_K(\theta,\delta)\, \delta\bigr)
&=\frac{1}{|K|_d}
  \frac{\partial B_K}{\partial t}(\theta,t_K(\theta,\delta))
  \frac{\partial t_K}{\partial\delta}(\theta,\delta)
=t_K(\theta,\delta).
\end{aligned}
\]
Continuity of the cutting level proves the asserted continuous
differentiability.
\end{proof}

\begin{proof}[Proof of Corollary~\ref{thm:buoyancy-density-interval}]
Fix $\theta\in\mathbb S^{d-1}$ and put
\[
F_\theta(\delta)
=\delta\bigl(s_P(\theta,\delta)-s_Q(\theta,\delta)\bigr).
\]
By Lemma~\ref{lem:buoyancy-density-differentiation},
$F_\theta\in C^1((0,1))$. Equality of the buoyancy surfaces gives
$F_\theta(\delta_j)=0$ for every $j$, so continuity implies
$F_\theta(\delta_0)=0$. Omitting at most one term, we may assume
$\delta_j\ne\delta_0$. Hence
\[
F_\theta'(\delta_0)
=\lim_{j\to\infty}
\frac{F_\theta(\delta_j)-F_\theta(\delta_0)}
     {\delta_j-\delta_0}
=0.
\]
The density-differentiation identity therefore yields
\[
t_P(\theta,\delta_0)-t_Q(\theta,\delta_0)
=F_\theta'(\delta_0)=0.
\]
Since $\theta$ was arbitrary and $\delta_0\ne\tfrac12$,
Theorem~\ref{floatingTHM} gives $P=Q$.
\end{proof}

\section{Fixed-density buoyancy data}
\label{sec:buoyancy-uniqueness}

Uniqueness from a buoyancy surface at a prescribed density
$\delta\in(0,1)\setminus\{\tfrac12\}$ remains unresolved by our
arguments for general convex polytopes. We describe the
ambiguity arising from complementary densities and the
centred section moments determined by the surface.
For polygons, we prove recovery of the difference body
normalized by the cube root of the area, yielding uniqueness
among centrally symmetric polygons at every density.

For a convex body $K\subset\mathbb R^d$, write
\[
C_{K,\delta}(\theta)
=K\cap H^+(\theta,t_K(\theta,\delta)),
\qquad
g_{K,\delta}(\theta)
=\frac{1}{\delta|K|_d}
 \int_{C_{K,\delta}(\theta)}x\,dx.
\]
The surface of centres
\[
\mathcal C_\delta K
=g_{K,\delta}(\mathbb S^{d-1})
\]
is the boundary of the Ulam floating body
$M_{\delta|K|_d}(K)$. Its unique supporting point with outer
normal $\theta$ is $g_{K,\delta}(\theta)$
\cite[Proposition~2.1]{HSW}. Thus equality of buoyancy surfaces
is equivalent to equality of their cap-centroid maps in every
direction.

\subsection{Complementary densities}
\label{subsec:complementary-densities}

The following relation explains why the density must be fixed
in a uniqueness question (see also \cite[Proposition 2.2]{HSW}).

\begin{proposition}
\label{prop:complementary-buoyancy}
Let $K\subset\mathbb R^d$ be a convex body with $g(K)=0$,
and let $\delta\in(0,1)$. Then
\begin{equation}
\label{eq:complementary-buoyancy}
\delta g_{K,\delta}(\theta)
+(1-\delta)g_{K,1-\delta}(-\theta)=0,
\qquad \theta\in\mathbb S^{d-1},
\end{equation}
and consequently
\[
\mathcal C_{1-\delta}K
=-\frac{\delta}{1-\delta}\mathcal C_\delta K.
\]
Moreover, the reflected dilate
\[
Q=-\frac{1-\delta}{\delta}K
\]
satisfies $g_{Q,1-\delta}(\theta)=g_{K,\delta}(\theta)$
for every $\theta$, and hence
$\mathcal C_{1-\delta}Q=\mathcal C_\delta K$.
If $\delta\ne\tfrac12$, then $Q\ne K$ and the densities differ.
At $\delta=\tfrac12$, the surface $\mathcal C_{1/2}K$ is
centrally symmetric about the origin, and $Q=-K$.
\end{proposition}

\begin{proof}
Uniqueness of the liquid level gives
\(
t_K(-\theta,1-\delta)=-t_K(\theta,\delta).
\)
Thus the corresponding caps partition $K$ up to their common
section, which has zero $d$-dimensional volume. Splitting the
first moment and using $g(K)=0$, we obtain
\[
\delta|K|_d\,g_{K,\delta}(\theta)
+(1-\delta)|K|_d\,g_{K,1-\delta}(-\theta)
=\int_K x\,dx=0.
\]
Dividing by $|K|_d$ proves
\eqref{eq:complementary-buoyancy}. Letting $\theta$ range over
$\mathbb S^{d-1}$ gives the asserted relation between the surfaces.
Put $\lambda=(1-\delta)/\delta$, so $Q=-\lambda K$.
The map $x\mapsto-\lambda x$ sends the
$(1-\delta)$-cap of $K$ with normal $-\theta$ onto the
$(1-\delta)$-cap of $Q$ with normal $\theta$. Hence
\[
g_{Q,1-\delta}(\theta)
=-\lambda g_{K,1-\delta}(-\theta)
=g_{K,\delta}(\theta),
\]
where the last equality follows from
\eqref{eq:complementary-buoyancy}.
Consequently, $\mathcal C_{1-\delta}Q=\mathcal C_\delta K$.

If $\delta\ne\tfrac12$, then $\lambda\ne1$ and
$|Q|_d=\lambda^d|K|_d\ne|K|_d$, so $Q\ne K$.
At $\delta=\tfrac12$, the surface relation becomes
$\mathcal C_{1/2}K=-\mathcal C_{1/2}K$, and $Q=-K$.
\end{proof}

\subsection{Section moments determined by the buoyancy surface}
\label{subsec:buoyancy-first-variation}

Fix $\delta\in(0,1)$. For a full-dimensional convex body
$K\subset\mathbb R^d$, $d\ge2$, put
\[
S_K(\theta)=K\cap H(\theta,t_K(\theta,\delta)),
\qquad c_K(\theta)=g(S_K(\theta)).
\]
The slicing argument in the proof of
Lemma~\ref{analyticity-lemma} applies to arbitrary convex
bodies and gives continuously differentiable cap volumes
and first moments at every interior cut.
For $u \perp \theta$, differentiating the fixed cap volume gives
$D_\theta t_K(\theta,\delta)[u]=c_K(\theta)\cdot u$.
Then, differentiating the first moment and using
$\int_{S_K(\theta)}(x-c_K(\theta))\,d\mathcal H^{d-1}(x)=0$
yields
\begin{equation}
\label{eq:buoyancy-first-variation}
D_\theta g_{K,\delta}(\theta)[u]
=A_{K,\delta}(\theta)u
=\frac{1}{\delta|K|_d}
\int_{S_K(\theta)}
(x-c_K(\theta))
\bigl((x-c_K(\theta))\cdot u\bigr)
\,d\mathcal H^{d-1}(x).
\end{equation}
Here $A_{K,\delta}(\theta):\theta^\perp\to\theta^\perp$
is the centred section second-moment operator normalized by
the cap volume. Since the section has nonempty relative
interior, for $0\ne u\perp\theta$,
\begin{equation*}
\label{eq:buoyancy-positive-variation}
u\cdot A_{K,\delta}(\theta)u
=\frac{1}{\delta|K|_d}
\int_{S_K(\theta)}
\bigl((x-c_K(\theta))\cdot u\bigr)^2
\,d\mathcal H^{d-1}(x)>0.
\end{equation*}
Thus $D_\theta g_{K,\delta}$ has rank $d-1$.
Formula~\eqref{eq:buoyancy-first-variation} is the centred-moment
form of the Third Theorem of Dupin; see
\cite[Theorem~2.11 and Corollary~2.12]{Za}.

\begin{proposition}
\label{prop:buoyancy-moment-characterization}
Let $P,Q\subset\mathbb R^d$, $d\ge2$, be convex bodies, and fix $\delta\in(0,1)$.
Their buoyancy surfaces are translates if and only if
\[
A_{P,\delta}(\theta)=A_{Q,\delta}(\theta)
\qquad\text{for every }\theta\in\mathbb S^{d-1}.
\]
Under these conditions, the surfaces are equal if and only if
$g_{P,\delta}(\theta_0)=g_{Q,\delta}(\theta_0)$
for one direction~$\theta_0$.
\end{proposition}

\begin{proof}
By \eqref{eq:buoyancy-first-variation} and connectedness of
$\mathbb S^{d-1}$, equality of the operators is equivalent to
constancy of $g_{P,\delta}-g_{Q,\delta}$.
Since $g_{K,\delta}(\theta)$ is the unique supporting point
with outer normal $\theta$, this is equivalent to translation
of the surfaces. Agreement in one direction makes the
translation vector zero.
\end{proof}

In the plane, writing $\ell_K(\theta)$ for the section length,
\eqref{eq:buoyancy-first-variation} (see \cite[Corollary~2.12]{Za}) reduces to
\begin{equation}
\label{A}
A_{K,\delta}(\theta)
=\frac{\ell_K(\theta)^3}{12\delta|K|_2}
\,\operatorname{Id}_{\theta^\perp}.
\end{equation}

\subsection{Planar recovery of the difference body}
\label{subsec:buoyancy-difference-body}

For polygons, straightness of the sides allows the buoyancy
curve to determine a normalized difference body.

\begin{proof}[Proof of Corollary \ref{diff-body}]
For
\(
\theta(\varphi)=(\cos\varphi,\sin\varphi),
u(\varphi)=(-\sin\varphi,\cos\varphi),
\)
and $K=P$ or $Q$, put
\[
h(\varphi)=t_K(\theta(\varphi),\delta),
\qquad
\ell_K(\varphi)
=|K\cap H(\theta(\varphi),h(\varphi))|_1.
\]
Equality of the buoyancy curves gives equality of their
unique supporting points with each outer normal. 
By \eqref{A}, we have
\[
\frac{|P|_2}{\ell_P^3}
=
\frac{|Q|_2}{\ell_Q^3}.
\]
Consequently,
\begin{equation}
\label{eq:buoyancy-chord-scaling}
\ell_Q=\lambda\ell_P,
\qquad
\lambda=\left(\frac{|Q|_2}{|P|_2}\right)^{1/3}.
\end{equation}

By Lemma~\ref{P_Dupin}, there is a common finite subdivision
of the parameter circle on whose open intervals the chord
endpoints lie in fixed side interiors. All subsequent
differentiations are on these intervals. Dupin's midpoint
identity gives the endpoints as
\[
x_K^\pm=h\theta+\left(h'\pm \frac1 2\ell_K\right)u.
\]
Putting $b_K=h+h''$, we obtain
\(
(x_K^\pm)'=(b_K\pm\frac 1 2 \ell_K')u\mp \frac 1 2\ell_K\theta.
\)
Since each endpoint moves along a straight line,
$\det((x_K^\pm)',(x_K^\pm)'')=0$.
Adding these two identities yields
\begin{equation}
\label{eq:buoyancy-flat-side-identity}
4b_K^2
=\ell_K\ell_K''-2(\ell_K')^2-\ell_K^2.
\end{equation}
Equations~\eqref{eq:buoyancy-chord-scaling} and
\eqref{eq:buoyancy-flat-side-identity} imply
$b_Q^2=\lambda^2b_P^2$. Substitution into the endpoint
derivatives therefore gives, as unordered pairs,
\[
\{(x_Q^+)',(x_Q^-)'\}
=
\varepsilon\lambda\{(x_P^+)',(x_P^-)'\},
\qquad \varepsilon\in\{-1,1\}.
\]
Although the sign may vary, for every unit vector $e$,
\begin{equation}
\label{eq:buoyancy-projection-variation}
|(x_Q^+)'\cdot e|+|(x_Q^-)'\cdot e|
=
\lambda\bigl(
|(x_P^+)'\cdot e|+|(x_P^-)'\cdot e|
\bigr).
\end{equation}

As in the construction of $T_{K,\delta}$ in
Lemma~\ref{lem:planar-flotation-reconstruction},
for either fixed endpoint sign, each boundary point
determines a unique oriented liquid chord. Hence each $x_K^\pm$ parametrizes $\partial K$ once.
At a vertex direction, the liquid line intersects
$\operatorname{int}K$ and is therefore transverse to the
incident sides. Since $h\in C^1$ by Lemma~\ref{analyticity-lemma},
the line--side intersection formulae and transversality give bounded one-sided endpoint derivatives.
Since there are finitely many vertex directions,
$x_K^\pm$ are Lipschitz and hence absolutely continuous.

For a unit vector $e$, write
$w_K(e)=\max_{x\in K}x\cdot e-\min_{x\in K}x\cdot e$
for the width of $K$ in direction $e$.
The projection of a convex boundary traversed once has
total variation $2w_K(e)$.
Indeed, fix either sign and set
\[
f(\varphi)=x_K^\pm(\varphi)\cdot e,
\qquad
m=\min_{x\in K}x\cdot e,
\qquad
M=\max_{x\in K}x\cdot e.
\]
Since $x_K^\pm$ traverses $\partial K$ once, convexity implies
that $f$ has two monotone portions between its minimum and
maximum. More precisely, regarding $f$ as $2\pi$-periodic,
choose $\alpha$ with $f(\alpha)=m$ and
$\beta\in(\alpha,\alpha+2\pi)$ with $f(\beta)=M$.
Then $f$ is nondecreasing on $[\alpha,\beta]$ and
nonincreasing on $[\beta,\alpha+2\pi]$. By periodicity and
absolute continuity,
\[
\begin{aligned}
\int_0^{2\pi}
\bigl|(x_K^\pm)'(\varphi)\cdot e\bigr|\,d\varphi&=\int_\alpha^{\alpha+2\pi}|f'(\varphi)|\,d\varphi
=\int_\alpha^\beta f'(\varphi)\,d\varphi
-\int_\beta^{\alpha+2\pi}f'(\varphi)\,d\varphi\\
&=(M-m)-(m-M)
=2w_K(e).
\end{aligned}
\]

Integrating~\eqref{eq:buoyancy-projection-variation} therefore gives
\[
\begin{aligned}
4w_Q(e)
&=\int_0^{2\pi}
\left(
\bigl|(x_Q^+)'(\varphi)\cdot e\bigr|
+\bigl|(x_Q^-)'(\varphi)\cdot e\bigr|
\right)\,d\varphi\\
&=\lambda\int_0^{2\pi}
\left(
\bigl|(x_P^+)'(\varphi)\cdot e\bigr|
+\bigl|(x_P^-)'(\varphi)\cdot e\bigr|
\right)\,d\varphi
=4\lambda w_P(e).
\end{aligned}
\]
The omitted vertex directions form a finite set and do not
affect these integrals. Since the support function of $K-K$
is $w_K$, we get
$Q-Q=\lambda(P-P)$.

If $P,Q$ are centrally symmetric, this identity makes $Q$
a translate of $\lambda P$. Taking areas gives
$|Q|_2=\lambda^2|P|_2$, whereas the definition of $\lambda$
gives $|Q|_2=\lambda^3|P|_2$. Hence $\lambda=1$.
Finally, translation covariance of the buoyancy curves and
their equality force the remaining translation to vanish.
\end{proof}

\medskip

\noindent{\textbf{Acknowledgment:}} We are grateful to Dmitry Ryabogin for inspiration and helpful discussions. ChatGPT was used in exploratory discussions and in the preparation of this manuscript.

\end{document}